\documentclass[12pt,a4paper]{article}
\usepackage[T2A]{fontenc}
\usepackage[cp1251]{inputenc}
\usepackage[english]{babel}
\usepackage{amsmath,amsfonts,amssymb, amsthm}
\usepackage{cite}
\usepackage{xcolor}

\theoremstyle{plain}
\newtheorem{theorem}{Theorem}[section]
\newtheorem{lemma}{Lemma}[section]

\theoremstyle{definition}

\newtheorem{example}{Example}[section]

\begin{document}

\title{On the residual of a factorized group with widely supersoluble factors}

\author{Victor S. Monakhov and Alexander A. Trofimuk\\
{\small Department of Mathematics and Programming Technologies,}\\
{\small Francisk Skorina Gomel State University,}\\
{\small Gomel 246019, Belarus}\\
{\small e-mail: victor.monakhov@gmail.com}\\
{\small e-mail: alexander.trofimuk@gmail.com}{\small $^{\dag}$ }
}

\date{}

\maketitle

{\bf Abstract.}
Let $\Bbb P$ be the set of all primes. A subgroup  $H$ of a group $G$ is called {\it $\mathbb P$-subnormal} in $G$, if either $H=G$, or there exists a chain of subgroups $H=H_0\le H_1\le \ldots \le H_n=G, \ |H_{i}:H_{i-1}|\in \Bbb P, \ \forall i.$ A group $G$ is called {\it widely supersoluble}, $\mathrm{w}$-supersoluble for short,
if every Sylow subgroup of $G$ is $\mathbb P$-subnormal in $G$. A group $G=AB$ with $\mathbb P$-subnormal $\mathrm{w}$-supersoluble subgroups $A$ and $B$ is studied.
The structure of its $\mathrm{w}$-supersoluble residual is obtained. In particular, it coincides with the nilpotent residual of the $\mathcal{A}$-residual of $G$. Here $\mathcal{A}$ is the formation of
all groups with abelian Sylow subgroups. Besides, we obtain new sufficient conditions for the $\mathrm{w}$-supersolubility of such group $G$.

\medskip


\medskip


\bigskip

\section*{Introduction}

Throughout this paper, all groups are finite and $G$ always denotes a finite group. We use the standard notations and terminology of~\cite{Hup}. The formations of all nilpotent, supersoluble  groups and groups with abelian Sylow subgroups are denoted
by $\mathfrak N$, $\mathfrak U$ and $\mathcal{A}$, respectively. The notation $Y\le X$
means that $Y$ is a subgroup of a group $X$ and $\mathbb P$~be the set of all primes. Let $\mathfrak X$ be a formation.
Then $G^{\mathfrak X}$ denotes the $\mathfrak
X$-residual of $G$.

By Huppert's Theorem~\cite[VI.9.5]{Hup},  a group $G$ is supersoluble if and only if for every proper subgroup $H$ of $G$ there exists a chain of subgroups
$$
H=H_0\le H_1\le \ldots \le H_n=G, \ |H_{i}:H_{i-1}|\in \Bbb P, \ \forall i.\eqno (1)
$$
So naturally the following definition.

A subgroup $H$ of a group $G$ is called $\mathbb P$-{\it subnormal}
in~$G$, if either~${H=G}$, or there is a chain subgroups~(1). We use the notation $H\mathbb ~Psn~G$.
This definition was proposed in~\cite{VVTSMJ10} and besides, in this paper  {\it $\mathrm{w}$-supersoluble} (widely supersoluble) groups, i.e. groups with $\mathbb P$-subnormal Sylow subgroups,  were investigated.
Denote by  $\mathrm{w}\frak U$ the class of all $\mathrm{w}$-supersoluble groups.

The factorizable groups $G=AB$ with $\mathrm{w}$-supersoluble factors $A$ and $B$ were investigated in~\cite{MC2017}, \cite{Mon_Trof_2018}, \cite{MT_Sib2019}, \cite{Vas_Tiut_SMZ_2012}. There are many other papers devoted to study factorizable groups, and the reader is referred to the book~\cite{Bal_Rom_As_2010} and the bibliography therein. A criteria for
$\mathrm{w}$-supersolvability was obtained by
A.\,F. Vasil'ev, T.\,I. Vasil'eva and V.\,N. Tyutyanov~\cite{Vas_Tiut_SMZ_2012}.

{\bf Theorem A.} \cite[Theorem 4.7]{Vas_Tiut_SMZ_2012} {\it Let $G = AB $ be a group which is the product of two $\mathrm{w}$-supersoluble subgroups~$A$ and~$B$. If $A$ and $B$ are $\Bbb P$-subnormal in $G$
and $G^\mathcal{A}$ is nilpotent, then $G$ is $\mathrm{w}$-supersoluble.}

We recall that two subgroups $A$ and $B$ of a group $G$ are said to be
{\it mutually $sn$-permutable} if $A$ permutes with all subnormal subgroups of $B$ and
$B$ permutes with all subnormal subgroups of $A$.
If $A$ and $B$ are mutually  $sn$-permutable subgroups of a group $G = AB$, then we say that $G$ is a {\it mutually
$sn$-permutable product} of $A$ and $B$, see~\cite{Carocca1998}.
In soluble groups, mutually $sn$-permutable factors are $\mathbb P$-subnormal~\cite[Lemma~4.5]{Vas_Tiut_SMZ_2012}.
The converse is not true, see the example~\ref{ex5} below.

A. Ballester-Bolinches, W.\,M.~Fakieh and M.\,C.~Pedraza-Aguilera~\cite{BalFakPed2019}  obtained the following results for the $sn$-permutable product  of the $\mathrm{w}$-supersoluble subgroups.

{\bf Theorem B.} {\it
Let $G=AB$ be the mutually $sn$-permutable product of subgroups $A$ and~$B$.
Then the following hold:

$(1)$ if  $A$ and $B$ are $\mathrm{w}$-supersoluble and
$N$~is a minimal normal subgroup of~$G$, then both $AN$ and $BN$ are $\mathrm{w}$-supersoluble,~\emph{\cite[Theorem~3]{BalFakPed2019}};

$(2)$ if  $A$ and $B$ are $\mathrm{w}$-supersoluble and
$(|A/A^\mathcal{A}|, |B/B^\mathcal{A}|) = 1$, then $G$ is
$\mathrm{w}$-supersoluble,~\emph{\cite[Theorem~5]{BalFakPed2019}}.
}

Present paper extends the Theorems~A and B.
We prove the following result.

{\bf Theorem 1.} {\it
Let $A$ and $B$~be  $\mathrm{w}$-supersoluble $\Bbb P$-subnormal subgroups of~$G$ and $G=AB$.  Then the following hold:

$(1)$ $G^{\mathrm{w}\frak U}=(G^\mathcal{A})^\frak N$;

$(2)$ if $N$~is a nilpotent  normal subgroup of~$G$, then both $AN$ and $BN$ are $\mathrm{w}$-supersoluble;

$(3)$ if $(|A/A^\mathcal{A}|, |B/B^\mathcal{A}|) = 1$, then $G$ is $\mathrm{w}$-supersoluble.
}

Theorem~A follows from assertion~(1) of Theorem~1. Theorem~B follows from assertions~(2) and~(3) of Theorem~1 since the group $G$ in Theorem~B is soluble.

\section{Preliminaries}\label{pre}

In this section, we give some definitions and basic results which
are essential in the sequel.
A group whose chief factors have prime orders is called {\it supersoluble}.
Recall that a {\it $p$-closed} group is a group with a normal Sylow $p$-subgroup and a {\it $p$-nilpotent} group is a group with a normal Hall $p^{\prime}$-subgroup.

Denote by $G^\prime $, $Z(G)$, $F(G)$ and $\Phi (G)$  the derived subgroup, centre, Fitting and Frattini  subgroups of~$G$ respectively.
We use~$E_{p^t}$  to denote an elementary abelian group of order~$p^t$ and $Z_m$ to denote a cyclic group of order~$m$. The semidirect product of a normal subgroup~$A$ and a subgroup~$B$ is written as follows: $A\rtimes B$.

Let $\mathfrak F$ be a formation.
Recall that the $\mathfrak F$-residual of $G$, that is the intersection of all those normal
subgroups $N$ of $G$ for which $G/N \in \mathfrak F$.
We define $\mathfrak X\mathfrak Y =\{G \in \mathfrak E \mid G^{\mathfrak Y}\in
\mathfrak X\}$ and call $\mathfrak X \mathfrak Y$ the
{\it formation product} of $\mathfrak X$ and $\mathfrak H$.
Here $\mathfrak E$ is the class of all finite groups.

If $H$ is a subgroup of $G$, then $H_G=\bigcap _{x\in G} H^x$ is called {\it the core} of $H$ in $G$.
If a group $G$ contains a maximal subgroup $M$ with trivial core, then $G$ is said to be {\it primitive} and $M$ is its {\it primitivator}.

A simple check proves the following lemma.

\begin{lemma} \label{l333} Let $\mathfrak{F}$ be a saturated formation and $G$ be a group. Assume that $G \notin\mathfrak{F}$, but $G/N \in \mathfrak{F}$ for all non-trivial normal subgroups $N$ of $G$. Then $G$ is a primitive group.
\end{lemma}

\begin{lemma} \label{l444} \emph{(\cite[Theorem II.3.2]{Hup})} Let  $G$ be a soluble primitive group and  $M$ is a primitivator of $G$. Then the following statements hold:

$(1)$ $\Phi (G)=1$;

$(2)$ $F(G)=C_G(F(G))=O_p(G)$ and  $F(G)$  is an elementary abelian subgroup of order~$p^n$ for some prime $p$ and some positive integer~$n$;

$(3)$ $G$ contains a unique minimal normal subgroup $N$ and moreover, $N=F(G)$;

$(4)$ $G=F(G)\rtimes M$ and $O_p(M)=1$.
\end{lemma}


\begin{lemma} \emph{(\cite[Proposition~2.2.8, Proposition~2.2.11]{Bal_Clas})} \label{l3}
Let $\mathfrak F$ and $\mathfrak H$  be formations, $K$ be normal in~$G$. Then the following hold:

$(1)$ $(G/K)^{\mathfrak F}=G^{\mathfrak F}K/K$;

$(2)$ $G^{\mathfrak F\mathfrak H}=(G^\mathfrak H)^\mathfrak F$;

$(3)$ if $\mathfrak H\subseteq \mathfrak F$, then
$G^\mathfrak F\le G^\mathfrak H$;

$(4)$ if $G=HK$, then $H^\mathfrak FK=G^\mathfrak FK$.
\end{lemma}

Recall that a group $G$ is said to be {\it siding} if every subgroup of the derived subgroup  $G^{\prime}$ is normal in $G$, see \cite[Definition~2.1]{Per}. Metacyclic groups,
t-groups (groups in which every subnormal subgroup is normal) are siding.
The group $G=(Z_6\times Z_2)\rtimes Z_2$  (IdGroup(G)=[24,8]) \cite{Gap} is siding, but not metacyclic and a t-group.

\begin{lemma} \label{l55}
Let $G$ be siding. Then the following hold:

$(1)$ if $N$ is normal in $G$, then $G/N$ is siding;

$(2)$ if $H$ is a subgroup of $G$, then $H$ is siding;

$(3)$ $G$ is supersoluble
\end{lemma}

\begin{proof}
1. By \cite[Lemma~I.8.3]{Hup}, $(G/N)^\prime=G^\prime N/N$. Let
$A/N$ be an arbitrary subgroup of $(G/N)^\prime$.
Then
$$
A\leq G^\prime N, \ A=A\cap G^\prime N=(A\cap G^\prime)N.
$$
Since  $A\cap G^\prime\leq G^\prime$, we have $A\cap G^\prime$ is normal in~$G$. Hence $(A\cap G^\prime)N/N$ is normal in~$G/N$.

2. Since $H\leq G$, it follows that $H^\prime\leq G^\prime$. Let  $A$ be an arbitrary subgroup of $H^\prime$. Then $A\leq G^\prime$ and $A$ is normal in $G$. Therefore $A$ is normal in  $H$.

3. We proceed by induction on the order of $G$. Let $N\leq G^\prime$ and
$|N|=p$, where $p$ is prime. By the hypothesis,  $N$ is normal in $G$.
By induction, $G/N$ is supersoluble and $G$ is supersoluble.
\end{proof}

\begin{lemma} \emph{(\cite[Lemma~ 3]{Mon_Kn_Res_2013})} \label{l4}
Let $H$~be a subgroup of $G$, and $N$~be a normal subgroup of~$G$.
Then the following hold:

$(1)$ if $N\le H$ and $H/N$ $\mathbb  Psn$ $G/N$, then $H$ $\mathbb  Psn$ $G$;

$(2)$ if $H$ $\mathbb  Psn$ $G$, then $(H\cap N)$ $\mathbb  Psn$ $N$,
$HN/N$ $\mathbb  Psn$ $G/N$ and $HN$ $\mathbb  Psn$~$G$;

$(3)$ if $H\le K\le G$, $H$ $\mathbb  Psn$ $K$ and  $K$ $\mathbb  Psn$ $G$, then
$H$ $\mathbb  Psn$~$G$;

$(4)$ if $H$ $\mathbb  Psn$ $G$, then $H^g$ $\mathbb  Psn$ $G$ for any $g\in G$.
\end{lemma}

\begin{lemma} \emph{(\cite[Lemma~4]{Mon_Kn_Res_2013})} \label{l4'}  Let $G$ be a soluble group, and $H$~be a subgroup of~$G$.
Then the following hold:

$(1)$ if $H$  $\mathbb Psn$ $G$ and $K\leq G$, then $(H\cap K)$
$\mathbb Psn$ $K$;

$(2)$ if $H_i$ $\mathbb Psn$ $G$, $i=1,2$, then $(H_1\cap H_2)$ $\mathbb Psn$ $G$.
\end{lemma}

\begin{lemma} \emph{(\cite[Lemma~5] {Mon_Kn_Res_2013})} \label{l66}
If $H$~is a subnormal subgroup of a soluble group~$G$, then $H$ is $\mathbb P$-subnormal in~$G$.
\end{lemma}

\begin{lemma}\label{l31} \emph{(\cite[Theorem~2.7]{VVTSMJ10})}
The class $\mathrm{w}\frak U$ is a hereditary saturated formation.
\end{lemma}

\begin{lemma}\label{l32}
$(1)$ If $G\in \mathrm{w} \mathfrak U$,
then~$G^\mathcal{A}$ is nilpotent,~\emph{\cite[Theorem~2.13]{VVTSMJ10}}.

$(2)$  $G\in \mathrm{w} \mathfrak U$ if and only if every  metanilpotent subgroup of $G$ is supersoluble,~\emph{\cite[Theorem~2.6]{Mon2016SMJ}}.

$(3)$  $G\in \mathrm{w} \mathfrak U$ if and only if $G$ has a Sylow tower of supersoluble type and every
biprimary subgroup of $G$  is supersoluble,~\emph{\cite[Theorem~B]{Mon_Kn_Res_2013}}.
\end{lemma}

\section{Factorizable groups with $\Bbb P$-subnormal \\ $\mathrm{w}$-supersoluble subgroups}\label{sec2}

\medskip

\begin{lemma}\label{l6} \emph{(\cite[Theorem 4.4]{Vas_Tiut_SMZ_2012})} Let $A$ and $B$ be $\mathbb  P$\nobreakdash-\hspace{0pt}subnormal subgroups of~$G$, and $G=AB$. If $A$ and $B$ have an ordered Sylow tower of supersoluble type, then $G$ has an ordered Sylow tower of supersoluble type.
\end{lemma}

{\bf Proof of Theorem~1\,(1).}
If~$G$ is  $\mathrm{w}$-supersoluble, then~$G^{\mathrm{w}\frak U}=1$ and $G^\mathcal{A}$ is nilpotent by Lemma~\ref{l32}\,(1).  Consequently
$G^{\mathrm{w}\frak U}=1=(G^\mathcal{A})^\frak N$
and the statement is true. Further, we assume that~$G$ is non-$\mathrm{w}$-supersoluble.
Since $\mathrm{w}\frak U\subseteq \mathfrak N \mathcal{A}$,
it follows that
$$
G^{(\mathfrak N \mathcal{A})}=(G^{\mathcal{A}})^\mathfrak N\le
G^{\mathrm{w}\frak U}
$$
by Lemma~\ref{l3}\,(2-3). Next we check the converse inclusion.
For this we prove that~$G/(G^\mathcal{A})^\frak N$ is
$\mathrm{w}$-supersoluble. By Lemma~\ref{l3}\,(1),
$$
(G/(G^\mathcal{A})^\frak N)^\mathcal{A} =
G^\mathcal{A} (G^\mathcal{A})^\frak N/(G^\mathcal{A})^\frak N =
G^\mathcal{A}/(G^\mathcal{A})^\frak N
$$
and  $(G/(G^\mathcal{A})^\frak N)^\mathcal{A}$
is nilpotent. The quotients
$$
G/(G^\mathcal{A})^\frak N=(A(G^\mathcal{A})^\frak N/(G^\mathcal{A})^\frak N)(B(G^\mathcal{A})^\frak N/(G^\mathcal{A})^\frak N,
$$
$$
A(G^\mathcal{A})^\frak N/(G^\mathcal{A})^\frak N\simeq A/A\cap (G^\mathcal{A})^\frak N,
$$$$ B(G^\mathcal{A})^\frak N/(G^\mathcal{A})^\frak N\simeq B/B\cap (G^\mathcal{A})^\frak N,
$$
hence the subgroups  $A(G^\mathcal{A})^\frak N/(G^\mathcal{A})^\frak N$ and
$B(G^\mathcal{A})^\frak N/(G^\mathcal{A})^\frak N$ are $\mathrm{w}$-supersoluble by Lemma~\ref{l31} and by Lemma~\ref{l4}\,(2),  they are  $\Bbb P$-subnormal in~$G/(G^\mathcal{A})^\frak N$.
By Theorem~A,  $G/(G^\mathcal{A})^\frak N$ is $\mathrm{w}$-supersoluble. \hfill $\square$

\begin{lemma} \label{t3} Let $G$ be a group, and  $A$ be a subgroup of $G$ such that ${|G:A|=p^\alpha}$, where $p\in \pi(G)$ and
$\alpha\in \mathbb{N}$. Suppose that $A$ is $\mathrm{w}$-supersoluble and  $\Bbb P$\nobreakdash-\hspace{0pt}subnormal in~$G$.
If $G$ is $p$-closed, then $G$ is $\mathrm{w}$-supersoluble.
\end{lemma}

\begin{proof}
Let $P$~be a Sylow $p$-subgroup of~$G$. Since $P$ is normal in~$G$ and $G=AP$, we have $G/P\simeq A/A\cap P\in \mathrm{w}\frak U$, in particular,
$G$ is soluble.
Because $G$ is soluble, it follows that $P$ is $\Bbb P$\nobreakdash-\hspace{0pt}subnormal in~$G$ by Lemma~\ref{l66}.
Let $Q$ be a Sylow $q$-subgroup of~$G$, $q\neq p$. Then $Q\leq A^x$ for some $x\in G$. By Lemma~\ref{l4}\,(4), $A^x$ is
$\Bbb P$\nobreakdash-\hspace{0pt}subnormal in~$G$. Since $A^x\in \mathrm{w}\frak U$, it follows that $Q$ is
$\Bbb P$\nobreakdash-\hspace{0pt}subnormal in~$A^x$ and $Q$ is $\Bbb P$\nobreakdash-\hspace{0pt}subnormal in~$G$ by Lemma~\ref{l4}\,(3). So,  $G$ is $\mathrm{w}$-supersoluble.
\end{proof}

\begin{lemma} \label{cc3_2} Let $A$ and $B$~be $\mathrm{w}$\nobreakdash-\hspace{0pt}supersoluble $\Bbb P$-subnormal subgroups of~$G$, and $G=AB$.
Suppose that $|G:A|=p^\alpha$, where $p\in \pi(G)$.
If $p$~is the greatest in~$\pi(G)$, then $G$ is $\mathrm{w}$\nobreakdash-\hspace{0pt}supersoluble.
\end{lemma}

\begin{proof}
Since every $\mathrm{w}$\nobreakdash-\hspace{0pt}supersoluble group has an ordered Sylow tower of supersoluble type, then by Lemma~\ref{l6}, $G$ has an ordered Sylow tower of supersoluble type. Hence $G$ is $p$-closed. By Lemma~\ref{t3}, we have that $G$ is $\mathrm{w}$\nobreakdash-\hspace{0pt}supersoluble.
\end{proof}

\begin{theorem} \label{th3_3}
Let $A$ be a $\mathrm{w}$-supersoluble $\Bbb P$-subnormal subgroup of~$G$,
and~$G=AB$.  Then~$G$ is $\mathrm{w}$-supersoluble  in each of the following cases:

$(1)$  $B$ is nilpotent and normal in~$G$;

$(2)$ $B$ is nilpotent and~$|G:B|$~is prime;

$(3)$ $B$ is normal in~$G$ and is a siding group.
\end{theorem}

\begin{proof}
We prove all three statements at the same time using induction on the order of $G$.
Note that $G$ is soluble in any case.
By Lemma~\ref{l66},  $B$ is $\Bbb P$-subnormal in~$G$ and~$G$ has an ordered Sylow tower of supersoluble type by Lemma~\ref{l6}.
If $N$~is a non-trivial normal subgroup of~$G$, then $AN/N$ is $\Bbb P$-subnormal in~$G/N$ by Lemma~\ref{l4}\,(2) and
$AN/N\simeq A/A\cap N$ is $\mbox {w}$-supersoluble by Lemma~\ref{l31}. The subgroup $BN/N\simeq B/B\cap N$ is nilpotent or a siding group by Lemma~\ref{l55}\,(1).
Hence $G/N=(AN/N)(BN/N)$ is $\mbox {w}$-supersoluble by induction.
Since the formation of all $\mbox {w}$-supersoluble groups is saturated by Lemma~\ref{l31}, we have $G$~is a primitive group by Lemma~\ref{l333}.
By Lemma~\ref{l444},    $F(G)=N=G_p$~is a unique minimal normal subgroup of~$G$ and $N=C_G(N)$, where $p$~is the greatest in~$\pi (G)$.

Since $A$ is $\Bbb P$-subnormal in~$G$, it follows that $G$ has a subgroup~$M$ such that $A\le M$ and $|G:M|$~is prime.
By Dedekind's identity, $M=A(M\cap B)$. The subgroup~$A$ is $\Bbb P$-subnormal in~$M$. The subgroup $M\cap B$ satisfies the requirements~(1)--(3). By induction,~$M$ is $\mbox {w}$-supersoluble.

1. If $B$ is nilpotent and normal in~$G$, then~$B=N$.
Hence $G=AN$ and  $A$ is a maximal subgroup of~$G$. Since $A$ is $\Bbb P$-subnormal in $G$, we have  $|G:A|=p=|N|$ and $G$ is supersoluble. Therefore $G$ is $\mbox {w}$-supersoluble.
So, in~(1), the theorem is proved.

2. Let $B$ be nilpotent and~$|G:B|=q$, where $q$~is prime.  Besides, let $|G:M|=r$, where $r$~is prime.
If $q\neq r$, then $(|G:M|,|G:B|)=1$. Since $G=MB$, $M$ and $B$~are $\Bbb P$-subnormal in~$G$ and $\mbox {w}$-supersoluble,
it follows obviously that $G$ is $\mbox {w}$-supersoluble. Hence $q=r$. If $q=p$, then $N$ is not contained in $M$. Thus $G=N\rtimes M$
and $|N|$~is prime.  Consequently $G$ is supersoluble and therefore $G$ is $\mbox {w}$-supersoluble. So, $q\neq p$. Then $G_p=N\leq M\cap B$. Since $B$ is nilpotent, $G_p=B\leq M$. Because $G=MB$, we have $G=M$, a contradiction. So, in~(2), the theorem is proved.

3. Let $B$ is normal in~$G$ and is a siding group.
If~$B$ is nilpotent, then~$G$ is $\mbox {w}$-supersoluble by~(1). Hence $B^{\prime}\ne 1$. Because $B^{\prime}$ is normal in~$G$ and nilpotent, we have
$N=B^{\prime}$. If $N$ is not contained in~$M$, then $G=N\rtimes M$ and $|N|$~is prime.  Consequently $G$ is supersoluble and therefore $G$ is $\mbox {w}$-supersoluble. Let $N$ be contained in $M$
and $N_1$~be a subgroup of prime order of~$N$ such that $N_1$ is normal in~$M$.
Then $N_1$ is normal in~$B$ by definition of siding group.
Hence $N_1$ is normal in~$G$. Consequently $G$  is $\mbox {w}$-supersoluble. So, in~(3), the theorem is proved.
\end{proof}

\medskip

{\bf Proof of Theorem~1\,(2).}

Note that by the Lemma~\ref{l6}, $G$ is soluble. By Theorem~\ref{th3_3}\,(1), Theorem~1\,(2) is true. \hfill $\square$

\medskip

{\bf Proof of Theorem~1\,(3).}
Assume that the claim is false and let $G$~be a minimal counterexample. By Lemma~\ref{l6},  $G$ has an ordered Sylow tower of supersoluble type.
If $N$~is a non-trivial normal subgroup of~$G$, then $AN/N$ and $BN/N$ are $\Bbb P$-subnormal in~$G/N$ by Lemma~\ref{l4}\,(2). Besides,
$AN/N\simeq A/A\cap N$ and  $BN/N\simeq B/B\cap N$ are $\mbox {w}$-supersoluble by Lemma~\ref{l31}.
By Lemma~\ref{l3}, we have
$$
(|(AN/N)/(AN/N)^\mathcal{A}|, |(BN/N)/(BN/N)^\mathcal{A}|) =
$$
$$=(|AN/(AN)^\mathcal{A}N|, |BN/(BN)^\mathcal{A}N|)=
$$
$$
=(|AN/A^\mathcal{A}N|, |BN/B^\mathcal{A}N|)=
\big{(}\frac{|A/A^\mathcal{A}|}{|S_1|},
\frac{|B/B^\mathcal{A}|}{|S_2|}\big{)},
$$
$$
S_1=(A\cap N)/(A^\mathcal{A}\cap N), \ S_2=(B\cap N)/(B^\mathcal{A}\cap N).
$$
Since  $(|A/A^\mathcal{A}|, |B/B^\mathcal{A}|) = 1$, it follows that
$$
(|(AN/N)/(AN/N)^\mathcal{A}|, |(BN/N)/(BN/N)^\mathcal{A}|) =1.
$$
The quotient $G/N=(AN/N)(BN/N)$ is $\mbox {w}$-supersoluble by induction.

Since the formation of all $\mbox {w}$-supersoluble groups is saturated by Lemma~\ref{l31}, we have $G$~is a primitive group by Lemma~\ref{l333}.
By Lemma~\ref{l444},    $F(G)=N=G_p$~is a unique minimal normal subgroup of~$G$ and $N=C_G(N)$,  where $p$~is the greatest in~$\pi (G)$.

By Lemma~\ref{t3}, $AN$ is $\mbox {w}$-supersoluble.
If $AN=G$, then $G$ is $\mbox {w}$-supersoluble, a contradiction.
Hence in the future we consider that $AN$ and $BN$~are proper subgroups of~$G$.

By Lemma~\ref{l32}\,(1),  $(AN)^\mathcal{A}$ is nilpotent.
Since $N=C_G(N)$, we have $(AN)^\mathcal{A}$~is a $p$-group. Because $AN/(AN)^\mathcal{A}\in \mathcal{A}$, it follows that all Sylow $r$-subgroups of $A$ are abelian, $r\neq p$. Since $A_p\leq G_p$, where~$A_p$~is a Sylow $p$-subgroup of~$A$, we have $A\in \mathcal{A}$. Similarly, $B\in \mathcal{A}$.
Hence $A^\mathcal{A}=1=B^\mathcal{A}$ and
$(|A|,|B|)=(|A/A^\mathcal{A}|, |B/B^\mathcal{A}|) = 1$. It is clear that $G$ is $\mathrm{w}$-supersoluble, a contradiction.
\hfill $\square$

\section{Examples}\label{sec6}

\medskip

The following example shows that for a soluble group $G=AB$ the  mutually $sn$-permutability  of subgroups $A$ and $B$ doesn't follow from $\Bbb P$-subnormality of these factors.

\begin{example} \label{ex5} {\rm The group $G=S_3\rtimes Z_3$ (IdGroup=[18,3])  has $\Bbb P$-subnormal  subgroups $A\simeq E_{3^2}$ and  $B\simeq Z_2$. However $A$ and $B$ are not mutually $sn$-permutable.}
\end{example}

The following example shows that we cannot omit the condition <<$G$ is $p$-closed>> in Lemma~\ref{t3}.

\begin{example} {\rm The group ${G=(S_3\times S_3)\rtimes Z_2}$ (IdGroup=[72,40]) has a $\Bbb P$-subnormal supersoluble subgroups $A\simeq Z_3\times S_3$. Besides ${|G:A|=2^2}$ and  Sylow 2-subgroup is maximal in $G$. Hence $G$ is non-$\mathrm{w}$-supersoluble.}
\end{example}

The following example shows that in Theorem~\ref{th3_3}\,(1) the normality of subgroup $B$ cannot be weakened to $\Bbb P$-subnormality.

\begin{example} \label{ex3_3} {\rm
The group $G=(Z_2 \times (E_{3^2}\rtimes Z_4))\rtimes Z_2$ (IdGroup=[144,115]) is non-$\mbox {w}$-supersoluble and  factorized by subgroups $A=D_{12}$ and $B=Z_{12}$. The subgroup $A$ has the chain of subgroups
$A<S_3\times S_3<Z_2\times  S_3 \times  S_3<G$ and $B$ has the  chain of subgroups
$B<Z_3 \times (Z_3\rtimes Z_4)<(Z_3 \times  (Z_3\rtimes Z_4))\rtimes Z_2<G$. Therefore $A$ and $B$ are $\Bbb P$-subnormal in~$G$.
}
\end{example}

The following example shows that in Theorem~\ref{th3_3}\,(2) it is impossible to weak the restrictions on the index of subgroup $B$.

\begin{example} {\rm
The alternating group $G=A_4$ is non-$\mbox {w}$-supersoluble and  factorized by subgroups $A=E_{2^2}$ and $B=Z_{3}$. It is clear that  $A$ is supersoluble and $\Bbb P$-subnormal in~$G$, and $B$ is nilpotent and $|G:B|=2^2$. The group $G=E_{5^2}\rtimes Z_3$  is non-$\mbox {w}$-supersoluble and has a nilpotent subgroup $Z_3$ of index~$5^2$. Therefore even for the greatest $p$ of $\pi(G)$, the index of $B$ cannot be equal~$p^\alpha$, $\alpha\geq 2$.}
\end{example}

The following example shows that in Theorem~\ref{th3_3}\,(3) the normality of subgroup $B$ cannot be weakened to subnormality.

\begin{example}
{\rm  The group $G=Z_3\times ((S_3\times S_3)\rtimes Z_2)$
(IdGroup=[216,157]) is non-$\mbox {w}$-supersoluble and  factorized by $\Bbb P$-subnormal supersoluble subgroup $A\simeq S_3\times S_3$
and  subnormal siding subgroup $B\simeq Z_3\times Z_3\times S_3$.}
\end{example}

\end{document}